\documentclass[11pt,twoside,reqno,centertags,english]{amsart}
\usepackage[oldstylenums]{kpfonts}
\usepackage{lettrine}
\usepackage[utf8]{inputenc}
\usepackage{babel}
\usepackage{bbding}
\usepackage{hyperref}
\usepackage{geometry}
\hypersetup{
colorlinks,
    citecolor=blue,
    filecolor=blue,
    linkcolor=blue,
    urlcolor=blue
    }

\usepackage{amsmath, amsfonts, amssymb}
\usepackage{amsthm}
\newtheorem{theorem}{Theorem}[section]
\newtheorem{lemma}[theorem]{Lemma}

\theoremstyle{remark}
\newtheorem{remark}[theorem]{Remark}
\theoremstyle{definition}
\newtheorem{definition}[theorem]{Definition}

\DeclareMathOperator{\grad}{grad}
\DeclareMathOperator*{\esslimsup}{ess\;lim\;sup}
\DeclareMathOperator*{\esssup}{ess\;sup}

\DeclareMathOperator{\Div}{div}

\newcommand{\entropy}{\mathcal E}
\newcommand{\entropyproduction}{D\mathcal E}
\newcommand {\R} {\mathbb{R}}
\newcommand{\cd}{\,d}
\newcommand{\problem}{\eqref{eq:p1}--\eqref{eq:p3}}
\newcommand{\assumptions}{\eqref{eq:g1}--\eqref{eq:g}}

\numberwithin{equation}{section}

\date{}

\title[Unbalanced Sobolev inequalities]{Convex Sobolev inequalities related to unbalanced optimal transport}

\author[S.~Kondratyev]{Stanislav Kondratyev}
\address[S.~Kondratyev]{CMUC, Department of Mathematics, University of 
Coimbra, 3001-501 Coimbra, Portugal}{}
\email{kondratyev@mat.uc.pt}
\author[D.~Vorotnikov]{Dmitry Vorotnikov}
\address[D.~Vorotnikov]{CMUC, Department of Mathematics, University of Coimbra, 3001-501 Coimbra, Portugal}{}
\email{mitvorot@mat.uc.pt}

\begin{document}
\begin{abstract}
We study the behaviour of various Lyapunov functionals (relative entropies) along the solutions of a family of nonlinear drift-diffusion-reaction equations coming from statistical mechanics and population dynamics. These equations can be viewed as 
gradient flows over the space of Radon measures equipped with the Hellinger-Kantorovich distance. The driving functionals of the gradient flows are not assumed to be geodesically convex or semi-convex.  We prove new 
isoperimetric-type functional inequalities, allowing us to control the 
relative entropies by their productions, which yields the exponential decay of the relative entropies. 
\end{abstract}
\maketitle


\vspace{10pt}
\textbf{MSC [2010] 26D10, 35K57, 35B40, 49Q20, 58B20}

\section{Introduction}
The unbalanced optimal transport \cite{LMS18,KMV16A,CP18,LMS16,CPSV18,Rez15} interpolates between the classical Monge-Kantorovich transport \cite{Vil03,Vil08} and the optimal information transport \cite{M15}. It equips the space of finite Radon measures with a formal Riemannian structure so that certain classes of reaction-diffusion equations and systems can be interpreted as gradient flows. This paper continues our investigation \cite{KMV16A,KMV16B,KMV17,KV17,KV19B} of such gradient flows and associated functional inequalities, see also \cite{CDM17,MG16,GLM17} for related studies. 

The class of PDEs that we consider in this paper is
\begin{align}
\partial_t \rho & = -\Div (\rho \nabla f) + f\rho, & (x, t) & \in \Omega 
\times (0, \infty),
\label{eq:p1}
\\
\rho \frac{\partial f}{\partial \nu} & = 0, & (x, t) & \in \partial \Omega 
\times (0, \infty),
\label{eq:p2}
\\
\rho & = \rho^0\ge 0, & (x, t) & \in \Omega \times {0}.
\label{eq:p3}
\end{align} Here $f=f(x,\rho(x,t))$ is a nonlinear function of $x$ and $\rho$ which is required to have a certain structure specified below in \eqref{f}, and $\Omega\subset\mathbb R^d$  is an open connected bounded domain admitting the relative isoperimetric 
inequality,  cf. \cite{Mazja},
\begin{equation}
\label{eq:iso}
P(A; \Omega) \ge C_\Omega \min(|A|^{\frac {d-1}d}, |\Omega \setminus A|^{\frac {d-1}d})
.\end{equation} All our results remain valid if $\Omega$ is a periodic box $\mathbb T^d$; in this case \eqref{eq:p2} is omitted.

The drift-diffusion-reaction equation \eqref{eq:p1} appears in statistical mechanics \cite{F04}. It also describes nonlinear fitness-driven models of population dynamics, cf. \cite{mc90,cosner05,cos13,T18,KV17}, where it is assumed that the 
dispersal strategy is determined by a local intrinsic characteristic of 
organisms called fitness. We refer to Section \ref{s:back} and to \cite{KV17} for more detailed discussions.

Let $g \colon (0, \infty) \to \mathbb R$ and $\psi \colon [0, \infty) \to 
\mathbb R$ be fixed $C^1$-smooth functions,  which satisfy the following assumptions:
\begin{gather}
g(1) = 0; \qquad g'(s) > 0 \ (s > 0),
\label{eq:g1}
\\
\psi(1) = 0, \qquad \psi(s) > 0 \ (s \ne 1),
\label{eq:psi1}
\\
\psi \in C^2(0, +\infty),\ \psi''(s) > 0 \ (s > 0, \ s \ne 1),
\label{eq:psipp}
\\
\lim_{s \to \infty} \psi'(x) = \infty,
\label{eq:infty2}
\\
|g(s)| + s |g'(s)|
\le h(s)
\quad
\text{a.~a. } s > 0; \ h \in L^1_\text{loc}[0, \infty)
,\label{eq:f1d}
\\
\ sg(s) \in C( [0, +\infty)).
\label{eq:g}
\end{gather}

Let $\rho_\infty \colon \overline \Omega \to \mathbb R$ be a fixed smooth strictly
positive function satisfying
\begin{equation}\label{eq:probm}
\int_\Omega \rho_\infty \, dx = 1.
\end{equation}

Define
\begin{equation} \label{f}
f=f(x,\rho(x)):= - g\left(\frac{\rho(x)}{\rho_\infty(x)}\right).
\end{equation}
Thus, the functions $g$ and $\rho_\infty$ determine the problem \problem, and the function $\psi$ is merely needed to define a Lyapunov functional for this problem,
\begin{equation}
0\le \entropy_\psi(\rho) := \int_\Omega \psi\left(\frac{\rho}{\rho_\infty}\right)
\rho_\infty \, dx,
\end{equation} which will be referred to as the relative entropy. Obviously, $\entropy_\psi(\rho)=0$ if and only if $\rho\equiv\rho_\infty$.
Formally calculating $\partial_t \entropy_\psi(\rho_t)$ along a solution 
of~\eqref{eq:p1}--\eqref{eq:p3} we obtain
\begin{equation*}
\partial_t \entropy_\psi(\rho_t) = - \entropyproduction_\psi(\rho_t)
,
\end{equation*}
where the entropy production $\entropyproduction_\psi$ is defined by
\begin{equation*}
\entropyproduction_\psi(\rho)
:= \int_\Omega
g'\left(\frac{\rho}{\rho_\infty}\right)
\psi''\left(\frac{\rho}{\rho_\infty}\right)
\left| \nabla \left(\frac{\rho}{\rho_\infty}\right) \right|^2
\rho \, dx
+ \int_\Omega g\left(\frac{\rho}{\rho_\infty}\right)
\psi'\left(\frac{\rho}{\rho_\infty}\right) \rho
\, dx
\end{equation*}

Setting
\begin{equation*}
r = \frac{\rho}{\rho_\infty}
,
\end{equation*}
we can write
\begin{gather}
\entropy_\psi(\rho) = \int_\Omega \psi(r) \, d\rho_\infty
\label{eq:epsirho}
\\
\entropyproduction_\psi(\rho)
=
\int_\Omega rg(r)
\psi'(r)
d\rho_\infty
+
\int_\Omega rg'(r)
\psi''(r)
| \nabla r|^2
d\rho_\infty
\label{eq:depsirho}
\end{gather}

Note that problem \problem ~can be viewed as a formal gradient flow (with respect to the unbalanced Hellinger-Kantorovich Riemannian structure) of the driving functional $\entropyproduction_{\psi_g}(\rho)$, where 
\begin{equation} \psi_g(s):=\int_1^s g(\xi)\,d\xi,\label{entcond}
\end{equation} see Section \ref{s:back} for the details.  
We are interested in the exponential decay of the Lyapunov functional \eqref{eq:epsirho} along the trajectories of this gradient flow. This is related to the entropy-entropy production inequalities of the form
\begin{equation}\label{eq:eepi} \entropy_\psi(\rho) \lesssim \entropyproduction_\psi(\rho).\end{equation} They can be viewed as unbalanced generalizations of the convex Sobolev inequalities \cite{AMTU01,BE85,J16}, see Section \ref{s:back}.

The main results of the paper are convex Sobolev inequalities akin to \eqref{eq:eepi}, see Theorems \ref{th:eep} and \ref{th:psi}, and existence and asymptotics of weak solutions to \problem, see Theorem \ref{th:ex}.

\section{Background and discussion} Assume for a while that $\Omega$ is a torus or is 
convex, although this is not required for our main results. The gradient of a scalar functional $\mathcal{E}$ on the space of finite Radon measures over $\overline\Omega$ with respect to the Hellinger-Kantorovich Riemannian structure (also known as the Wasserstein-Fisher-Rao one)  was calculated 
in~\cite{KMV16A,LMS16}:
\begin{equation*}
\grad_{HK} \mathcal{E}(\rho) =-\Div \left(\rho \nabla \frac{\delta \mathcal 
        E}{\delta \rho}\right) + u\frac{\delta \mathcal E}{\delta \rho} 
.
\end{equation*} The first term on the right-hand side is the Otto-Wasserstein gradient $\grad_{W} \mathcal{E}(\rho)$, cf. \cite{otto01,Vil03}, and the second one is the Hellinger-Fisher-Rao gradient $\grad_{H} \mathcal{E}(\rho)$, cf. \cite{KLMP}.  It is easy to compute that $\frac{\entropyproduction_{\psi_g}(\rho)}{\delta \rho}=-f(x,\rho)$, hence \problem~ may be interpreted as a gradient flow: \begin{equation}
\partial_t \rho=-\grad_{HK} \entropyproduction_{\psi_g}(\rho),\quad \rho(0)=\rho^0.
\label{eq:gradf}
\end{equation} The production of the relative entropy $\entropy_\psi(\rho)$ along the Otto-Wasserstein gradient flow \begin{equation}
\partial_t \rho=-\grad_{W} \entropyproduction_{\psi_g}(\rho)
\label{eq:gradf1}
\end{equation} is $$\entropyproduction^W_\psi(\rho)
:=
\int_\Omega rg'(r)
\psi''(r)
| \nabla r|^2
d\rho_\infty.$$ Similarly, the production of the same entropy along the Hellinger gradient flow \begin{equation}
\partial_t \rho=-\grad_{H} \entropyproduction_{\psi_g}(\rho)
\label{eq:gradf2}
\end{equation} is $$\entropyproduction^H_\psi(\rho)
:=
\int_\Omega rg(r)
\psi'(r)
d\rho_\infty.
$$ 

In the case of non-convex $\Omega$ we can abuse the terminology and still refer to  
\problem~ as to a gradient flow. 

It is clear that $$\entropyproduction^W_\psi(\rho)+\entropyproduction^H_\psi(\rho)=\entropyproduction_\psi(\rho).$$ Generally speaking, neither the Otto-Wasserstein nor the Fisher-Rao entropy production are able to control the relative entropy, so \eqref{eq:eepi} is a result of an interplay between the reaction, diffusion and drift. A simple counterexample to \begin{equation}\label{eq:frao}\entropy_\psi(\rho) \lesssim \entropyproduction_\psi^H(\rho)\end{equation} is $\rho_\infty 1_{A}$ with $A$ being a proper subset of $\Omega$. Indeed, $\entropyproduction_\psi^H(\rho_\infty 1_{A})=0$ due to \eqref{eq:g1}, \eqref{eq:f1d} and \eqref{eq:g}. 
 It is easy to construct a smooth example by mollifying this one. A trivial counterexample to\begin{equation}\label{eq:wasss}\entropy_\psi(\rho) \lesssim \entropyproduction_\psi^W(\rho)\end{equation} is $k\rho_\infty$ where $k\neq 1$ is a non-negative constant.  
 
\begin{remark} Note that the two counterexamples intersect at $\rho\equiv 0$, which violates our target inequality \eqref{eq:eepi}. However, we will observe, cf. Theorems \ref{th:eep} and \ref{th:psi}, that it suffices keep the total mass $\int_\Omega \rho$ bounded away from $0$ to secure \eqref{eq:eepi}. \end{remark}

In view of \eqref{eq:probm}, in order to obtain more interesting and instructive examples we should restrict ourselves to probability densities $\rho$. The sequence $$\rho_n=\rho_\infty\frac n {n-1}1_{(\frac 1 n,1)}$$ of probability densities on $\Omega=(0,1)$ is a counterexample to \eqref{eq:frao}. Indeed, the left-hand side of \eqref{eq:frao} is of order $n^{-1} $ and the right-hand side is $\lesssim n^{-2}$.

Inequality \eqref{eq:wasss} for $\int_\Omega \rho=1$ deserves a more detailed discussion. 

Let us start with considering $g(s)=\log s$. In this case, as first observed in the seminal paper \cite{JKO}, the gradient flow \eqref{eq:gradf1} is the linear Fokker-Planck equation, and the celebrated Bakry-\'Emery approach allows one to prove \eqref{eq:wasss} for $\Omega=\R^d$ \cite{AMTU01,BE85,J16}. However, it is crucial to have concavity of $\frac 1 {\psi''(s)}$, which we never assume in this work. These instances of \eqref{eq:wasss} are referred to as \emph{convex Sobolev inequalities}, which inspired the title of our paper. The particular case \begin{equation}\label{beckf} \psi(s)
 =
\begin{cases}
\frac 1{p(p- 1)}\left(  s^p-ps + p-1 \right), & \text{if } 1< p\le 2
\\
s\log s -s + 1, & \text{if } p = 1
\end{cases}\end{equation} implies the log-Sobolev inequality for $p=1$, the Poincar\'e inequality for $p=2$ and Beckner's inequalities \cite{B89} for $1<p<2$. Namely,  \eqref{eq:wasss} may be rewritten as
\begin{equation}\label{eq:wassb}\int_\Omega r^p \, d\rho_\infty - \left(\int_\Omega r\, d\rho_\infty\right)^p\lesssim
\int_\Omega 
r^{p-2}
| \nabla r|^2
d\rho_\infty,  \quad 1<p\le 2.\end{equation} In contrast, our assumptions on $\psi$ admit any $p>2$ in \eqref{beckf}, which yields the following ``Beckner-Hellinger inequality'': \begin{multline}\label{eq:hkb}\int_\Omega r^p \, d\rho_\infty - \left(\int_\Omega r\, d\rho_\infty\right)^p\lesssim
\int_\Omega 
r^{p-2}
| \nabla r|^2
d\rho_\infty\\+\int_\Omega r\log\left(\frac r {\int_\Omega r\, d\rho_\infty}\right)
\left(r^{p-1}- \left(\int_\Omega r\, d\rho_\infty\right)^{p-1}\right)
d\rho_\infty, \quad p>2.\end{multline}

Consider now the case $g(s)=\frac {s^{\alpha-1}-1}{\alpha-1}$, $\alpha> 0$, $\alpha\neq 1$. Assume for simplicity that $|\Omega|=1$ and $\rho_\infty\equiv 1$. Then \eqref{eq:gradf1} is the porous medium equation, cf. \cite{otto01}.  The alleged inequality \eqref{eq:wasss} for the relative entropy \eqref{beckf}, $p\in(1,\infty)$, reads
\begin{equation}\label{eq:wassal}\int_\Omega r^p  - \left(\int_\Omega r\right)^p\lesssim
\left(\int_\Omega r\right)^{1-\alpha} \int_\Omega 
r^{p+\alpha-3}
| \nabla r|^2.\end{equation} 
Setting $q:=\frac {2p}{p+\alpha-1}$, $l:=\frac {p+\alpha-1}{2}$, $u:=r^{l}$, we rewrite \eqref{eq:wassal} in the form
\begin{equation}\label{eq:wassal1}\int_\Omega u^q  - \left(\int_\Omega u^{1/l}\right)^{lq}\lesssim {\left(\int_\Omega u^{1/l}\right)^{l(q-2)}} 
\int_\Omega 
| \nabla u|^2.\end{equation} The inequality \begin{equation}\label{eq:wassal2}\int_\Omega u^q  - \left(\int_\Omega u^{1/l}\right)^{lq}\lesssim
\left(\int_\Omega 
| \nabla u|^2\right)^{q/2}.\end{equation}
similar to \eqref{eq:wassal1} appears in \cite{CJS13}, see also \cite{CDGJ,DGG08}. It holds for $0<q<2,$ $lq> 1$, that is, for $\alpha>1$, $p> 1$. Assume for a moment that the the relative entropy, i.e., the left-hand side of \eqref{eq:wassal2}, is a priori bounded. Since $ql \ge 1$, the mass $\int_\Omega u^{1/l}$ is a priori bounded. Consequently, \eqref{eq:wassal2} is weaker than \eqref{eq:wassal1} since the exponent $q/2$ is less than $1$, and it is plausible that \eqref{eq:wassal1} cannot be true. Inequality \eqref{eq:wassal2} for $q=2$ is equivalent to Beckner's inequality \eqref{eq:wassb}. As explained in \cite{DGG08}, inequality \eqref{eq:wassal2} is wrong for $q>2$. In this connection, our results yield the following variant of \eqref{eq:wassal1}:
\begin{multline}\label{eq:wassal4}\int_\Omega u^q  - \left(\int_\Omega u^{1/l}\right)^{lq}\lesssim {\left(\int_\Omega u^{1/l}\right)^{l(q-2)}} 
\int_\Omega 
| \nabla u|^2\\+\left(\int_\Omega u^{1/l}\right)^{l(q-2)}\int_\Omega u^{1/l}\left(\frac {u^{(\alpha-1)/l}-\left(\int_\Omega u^{1/l}\right)^{\alpha-1}}{\alpha-1}\right)
\left(u^{(p-1)/l}- \left(\int_\Omega u^{1/l}\right)^{p-1}\right)\end{multline} for any $q>0$, $q\neq 2$, $1<lq<1+2l$, that is, any $\alpha>0$, $\alpha\neq 1$, $p>1$.

The counterparts of the alleged inequalities \eqref{eq:wassal} and \eqref{eq:wassal1} for $p=1$ are
\begin{equation}\label{eq:wassalp1}\int_\Omega r\log\left(\frac r {\int_\Omega r}\right)\lesssim
\left(\int_\Omega r\right)^{1-\alpha} \int_\Omega 
r^{\alpha-2}
| \nabla r|^2,\end{equation}
\begin{equation}\label{eq:wassal1p1}\int_\Omega u^q\log\left(\frac {u^q} {\int_\Omega u^q}\right)\lesssim {\left(\int_\Omega u^{q}\right)^{\frac{q -2}q}} 
\int_\Omega 
| \nabla u|^2.\end{equation} Here $q=\frac {2}{\alpha}$. This resembles the inequality \begin{equation}\label{eq:wassal2p1}\int_\Omega u^q\log\left(\frac {u^q} {\int_\Omega u^q}\right)\lesssim \left(
\int_\Omega 
| \nabla u|^2\right)^{q/2},\quad q<2,\end{equation} which was established in \cite{CDGJ,DGG08}. Since $q/2<1$, \eqref{eq:wassal2p1} is weaker than \eqref{eq:wassal1p1}, so it seems that \eqref{eq:wassal1p1} cannot be true. Our results imply the following variant of \eqref{eq:wassal1p1}:
\begin{multline}\label{eq:wassal4p1}\int_\Omega u^q\log\left(\frac {u^q} {\int_\Omega u^q}\right)\lesssim {\left(\int_\Omega u^{q}\right)^{\frac{q -2}q}} 
\int_\Omega 
| \nabla u|^2\\+{\left(\int_\Omega u^{q}\right)^{\frac{q -2}q}}\int_\Omega u^q\log\left(\frac {u^q} {\int_\Omega u^q}\right)\left(\frac {u^{2-q}-\left(\int_\Omega u^{q}\right)^{\frac 2 q -1}}{2-q}\right), \quad  q>0,\, q\neq 2.\end{multline}

\begin{remark}Inequalities \eqref{eq:hkb}, \eqref{eq:wassal4}, \eqref{eq:wassal4p1} are obtained assuming $\int_\Omega r\, d\rho_\infty=1$ (so that \eqref{eq:eep-a} is automatically satisfied), but hold without this normalization due to their homogeneity.\end{remark}

Many authors studied \eqref{eq:wasss} or related inequalities in the particular case $\psi=\psi_g$, that is, when the driving entropy is compared to its production, cf., e.g.,  \cite{otto01,Vil03,Vil08,AGS06,CJM01}. In this connection, the strict geodesic convexity of the driving entropy normally plays the pivotal role. In \cite{KV17} (see also \cite{KMV16A}) we studied \eqref{eq:eepi} for $\psi=\psi_g$ without assuming neither Otto-Wasserstein nor Hellinger-Kantorovich geodesic convexity (we also never assume any similar condition in the present paper).   The inequalities obtained there can be further refined \cite{KV19B} be means of studying gradient flows in the spherical Hellinger-Kantorovich space \cite{LM17,BV18}, which is beyond the scope of the present paper (though it may seem strange, even non-negativity of the entropy production is uncertain for the spherical Hellinger-Kantorovich flows in the case $\psi\neq\psi_g$).  The proofs in the present paper are more direct and simple than in \cite{KV17} due to the  ``quasihomogeneous structure'' \eqref{f}.  

Our last example concerns $g(s)=\frac 12 \log \frac{2s^2}{1 +s^2}$, which corresponds to the arctangential heat equation \cite{Bre19}. The relative entropy $\entropy_{\psi_g}$ generated by this $g$ is geodesically convex  neither in the Otto-Wasserstein nor in the Hellinger-Kantorovich sense, cf. \cite{KV19B}. Take $\psi(s)=s\log s-s +1$.  Then we infer the following inequality resembling the log-Sobolev one: \begin{equation}\label{eq:hkbtan}\int_\Omega (r\log r-r+1) \, d\rho_\infty \lesssim
\int_\Omega 
\frac 1 {r(1+r^2)}
| \nabla r|^2
d\rho_\infty\\+\int_\Omega r\log r\left(\log \frac{2r^2}{1 + r^2}\right)
d\rho_\infty\end{equation} provided $\int_\Omega r\, d\rho_\infty$ is bounded away from $0$.

Nonlinear Fokker-Planck equations akin to \eqref{eq:gradf1} model behaviour of various stochastic systems, see \cite{F05,Ts09,J16,BLMV}. The related
drift-diffusion-reaction equation \eqref{eq:p1} was suggested in \cite{F04}. On the other hand, equation \eqref{eq:p1} belongs to the class of nonlinear models (cf. \cite{cos13,T18,XBF17,KV17,KV19B,mc90,cosner05}) for the 
spatial dynamics of populations which are tending to achieve the \emph{ideal 
free distribution} \cite{FC69,fr72} (the distribution which happens if 
everybody is free to choose its location) in a heterogeneous environment. The 
dispersal strategy is determined by a local intrinsic characteristic of 
organisms called \emph{fitness}. The 
fitness manifests itself as a growth rate, and simultaneously affects the 
dispersal as the species move along its gradient towards the most favorable 
environment. In \eqref{eq:p1}, $\rho(x,t)$ is the density of organisms, and 
$f(x,\rho)$ is the fitness.  The equilibrium $\rho\equiv \rho_\infty$ when the fitness 
is constantly zero corresponds to the ideal free distribution. The works \cite{CW13,ccl13,ltw14,XBF17,KMV16A,KMV16B,KMV17,KV17} perform mathematical analysis of some of such fitness-driven models. Our Theorem 
\ref{th:ex} indicates that the populations converge to the ideal free 
distribution with an exponential rate.

\label{s:back}

\section{Main results}

We start by introducing the weak solutions to~\eqref{eq:p1}--\eqref{eq:p3}, following the lines of~\cite{KV17,KV19B}.

Define
\begin{equation*}
G(s) = \int_0^s \xi g'(\xi) \cd \xi \qquad (s \ge 0)
,
\end{equation*}
where the integral exists by~\eqref{eq:f1d}.  Observe that
\begin{equation*}
G'(s) = s g'(s) > 0, \quad (s > 0); \qquad G(0) = 0,
\end{equation*}
so that $G$ is a nonnegative continuous increasing function on $[0, \infty)$.

Set
\begin{equation*}
\Phi(x, u) = \rho_\infty(x) G\left(\frac{u}{\rho_\infty(x)}\right), \quad u\geq0
.
\end{equation*}
As in~\cite{KV17}, we can write~\eqref{eq:p1} in the form
\begin{equation}
\label{eq:p1Phi}
\partial_t \rho = \Delta \Phi - \Div(\Phi_x + \rho f_x) + \rho f
,
\end{equation}
where~$\Phi$ stands for~$\Phi(x, \rho(x, t))$.
\begin{definition}
Let $\rho^0 \in L^\infty(\Omega)$; $Q_T:=\Omega\times(0,T)$.  A function~$\rho \in L^\infty(Q_T)$ is
called a \emph{weak solution} of~\eqref{eq:p1}--\eqref{eq:p3} on $[0, T]$ if 
for $r = \rho/\rho_\infty$ we have $G(r(\cdot)) \in L^2(0, T; H^1(\Omega))$ 
and
\begin{equation}
\label{eq:def1}
\int_0^T
\int_\Omega
(\rho \partial_t \varphi + ( - \nabla \Phi + \Phi_x + \rho f_x) \cdot \nabla 
\varphi + f \rho \varphi) \cd x \cd t
=
\int_\Omega \rho^0(x) \varphi(x, 0) \cd x
\end{equation}
for any function $\varphi \in C^1(\overline \Omega \times [0, T])$ such that 
$\varphi(x, T) = 0$.
A function $\rho \in L^\infty_\text{loc}([0, \infty); L^\infty(\Omega))$ is 
called a \emph{weak solution} of \eqref{eq:p1}--\eqref{eq:p3} on $[0, \infty)$ 
if for any~$T > 0$ it is a weak solution on~$[0, T]$.
\end{definition}

\begin{remark}
\label{rem:l2}
For $\rho \in L^\infty(Q_T)$ we automatically have $G(r) \in L^\infty(Q_T)$, 
so the condition $G(r(\cdot)) \in L^2(0, T; H^1(\Omega))$ is equivalent to 
$rg'(r)\nabla r \in L^2(Q_T)$.  Here $r = \rho/\rho_\infty$.
\end{remark}

Formally, the integrand $ rg'(r) \psi''(r) | \nabla r|^2$ vanishes 
if~$r = 0$.  Otherwise it can be written as
\begin{equation*}
rg'(r) \psi''(r) | \nabla r|^2
=
\frac{1}{r} \frac{\psi''(r)}{g'(r)} |rg'(r)\nabla r|^2
=
\frac{1}{r} \frac{\psi''(r)}{g'(r)} |\nabla G(r)|^2
.
\end{equation*}
This motivates the following extension of the entropy production suitable for 
weak solutions.
\begin{definition}
If $\rho \in L^\infty(\Omega)$ and $G(r) \in H^1(\Omega)$, then the 
\emph{entropy production} is defined by
\begin{align}
\entropyproduction_\psi(\rho)
&=
\int_\Omega rg(r)
\psi'(r)
d\rho_\infty
+
\int_{[r > 0]}
rg'(r)
\psi''(r)
| \nabla r|^2
d\rho_\infty
\notag
\\&
\equiv
\int_\Omega rg(r)
\psi'(r)
d\rho_\infty
+
\int_{[r > 0]}
\frac{1}{r} \frac{\psi''(r)}{g'(r)} |\nabla G(r)|^2
d\rho_\infty
.
\label{eq:defdegen}
\end{align}
\end{definition}
\begin{remark}
Observe that although the integrand with the gradient in~\eqref{eq:defdegen} 
is a nonnegative measurable function on~$\Omega$, the integral, and hence the 
entropy production, may be infinite.
\end{remark}

The following entropy-entropy production inequality applicable to weak 
solutions is based on an isoperimetric-type inequality established in 
Section~\ref{sec:ineq}.

\begin{theorem}[Entropy-entropy production inequality]
\label{th:eep}
Suppose that~$g$ and~$\psi$ satisfy~\assumptions.  Let $U \subset 
L_+^\infty(\Omega)$ be a set of functions such that for any $\rho \in U$ and 
$r = \rho/\rho_\infty$, we have $G(r) \in H^1(\Omega)$ and
\begin{gather}
\label{eq:eep-a}
\inf_{\rho \in U} \| \rho \|_{L^1(\Omega)} > 0
,
\\
\sup \{ \entropy_\psi(\rho) \colon \rho \in U \} < \infty
. \label{eq:eep-aa}
\end{gather}
Then there exists $C_U$ such that
\begin{equation}
\label{eq:eep-b}
\entropy_\psi(\rho) \le C_U \entropyproduction_\psi(\rho) \quad (\rho \in U)
.
\end{equation}
\end{theorem}
\begin{proof}
The idea is to use the isoperimetric-type inequality provided by 
Theorem~\ref{th:psi} (see Section~\ref{sec:ineq}).  Since we are dealing with 
a less regular setting at the moment, we argue by approximation.

Take $\rho \in U$ and as usual, put $r = \rho/\rho_\infty$.  Arguing as 
in~\cite[proof of Theorem~1.7]{KV17}, we see that there exists a sequence of 
functions $G_n \in C(\overline \Omega) \cap C^\infty(\Omega)$ taking values 
in~$(0, a)$, where $a < G(\infty)$, such that
\begin{equation*}
G_n \to G(r(\cdot)) \qquad \text{in $H^1$ and a.~e.\ in~$\Omega$}
.
\end{equation*}
Set $r_n(x) = G^{-1}(G_n(x))$ and $\rho_n(x) = r_n(x) \rho_\infty(x)$, so that 
$G_n(x) = G(r_n(x))$. Clearly, $r_n$ and~$\rho_n$ are positive and reasonably 
smooth, the sequences $\{r_n\}$ and $\{\rho_n\}$ are bounded in 
$L^\infty(Q_T)$ (specifically, the former is bounded by~$G^{-1}(a)$), and by 
the continuity of~$G^{-1}$ we have
\begin{gather*}
r_n \to r, \quad \rho_n \to \rho \text{ a.~e.\ in $\Omega$}
.
\end{gather*}
In particular, this implies that~$\rho_n$ converges to~$\rho$ 
in~$L^1(\Omega)$.  Further, by the Lebesgue Dominated Convergence we have
\begin{equation}
\entropy_\psi(\rho_n) \to \entropy_\psi(\rho)
.
\label{eq:tmp7}
\end{equation}
Thus, if we denote the infimum in~\eqref{eq:eep-a} by $d_U$ and the supremum 
in~\eqref{eq:eep-aa} by $E_U$, there is no loss of generality in assuming that 
$\|\rho_n\|_{L^1(\Omega)} \ge d_U/2$ and $\entropy_\psi(\rho_n) \le 2E_U$.  It 
follows from Theorem~\ref{th:psi} that there exist~$C$ and~$\sigma$ both 
depending on $d_U$ and $E_U$ (but not on the approximation nor on~$\rho$ 
itself) such that
\begin{equation}
\label{eq:tmp6}
\entropy_\psi(\rho_n) \le
C
\left(
\int_\Omega r_n g(r_n) \psi'(r_n) \cd \rho_\infty
+
\int_{[r \ge \sigma]}
r_n g'(r_n) \psi''(r_n) |\nabla r_n|^2 \cd \rho_\infty
\right)
.
\end{equation}

By the Lebesgue Dominated Convergence we have
\begin{equation}
\int_\Omega r_n g(r_n) \psi'(r_n) \cd \rho_\infty
\to
\int_\Omega r g(r) \psi'(r) \cd \rho_\infty
\label{eq:tmp8}
.
\end{equation}
Further, we have
\begin{equation*}
\int_{[r_n \ge \sigma]} r_n g'(r_n) \psi''(r_n) |\nabla r_n|^2 \cd \rho_\infty
=
\int_\Omega
1_{[r_n \ge \sigma]}
\frac{\psi''(r_n)}{r_n g'(r_n)}
|\nabla G_n|^2
\cd \rho_\infty
.
\end{equation*}
On one hand, $\nabla G_n \to \nabla G$ in $L^2(\Omega)$.  On the other hand, 
the functions
\begin{equation*}
h_n =
1_{[r_n \ge \sigma]}
\frac{\psi''(r_n)}{r_n g'(r_n)}
\end{equation*}
are uniformly bounded in $L^\infty(\Omega)$, and since we obviously have
\begin{equation*}
\limsup_{n \to \infty} 1_{[r_n \ge \sigma]} \le 1_{[r \ge \sigma]}
\qquad \text{a.~e.\ in $\Omega$}
,
\end{equation*}
we also have
\begin{equation*}
\limsup_{n \to \infty}
h_n(x)
\le
1_{[r \ge \sigma]}
\frac{\psi''(r)}{r g'(r)}
\qquad \text{a.~e.\ in $\Omega$}
.
\end{equation*}
Using Reverse Fatou's Lemma for products (Lemma~\ref{lem:fatou} in the Appendix), we 
obtain
\begin{align*}
\limsup_{n\to\infty}
\int_{[r_n \ge \sigma]}
r_n g'(r_n) \psi''(r_n) |\nabla r_n|^2 \cd \rho_\infty
& =
\limsup_{n\to\infty}
\int_\Omega
h_n |\nabla G_n|^2 \cd \rho_\infty
\\ &
\le
\int_\Omega
1_{[r \ge \sigma]}
\frac{\psi''(r)}{r g'(r)}
|\nabla G|^2 \cd \rho_\infty
\\ &
\le
\int_{[r > 0]}
r g'(r) \psi''(r) |\nabla r|^2 \cd \rho_\infty
.
\end{align*}
Combining this with \eqref{eq:tmp7} and \eqref{eq:tmp8}, we see that we can 
pass to the limit in~\eqref{eq:tmp6} and obtain~\eqref{eq:eep-b} with $C_U = 
C$.
\end{proof}

\begin{theorem}[Existence and asymptotics of weak solutions]
\label{th:ex}
Assume \assumptions.  Then for any~$\rho^0 \in L^\infty_+(\Omega)$ there 
exists a nonnegative weak solution $\rho \in L^\infty(\Omega \times (0, 
\infty))$ of problem \eqref{eq:p1}--\eqref{eq:p3} which enjoys the following properties:
\begin{enumerate}
\item
$\rho$ satisfies the entropy dissipation inequality in the sense of measures: 
for any smooth nonnegative compactly supported function $\chi \colon (0, T) 
\to \mathbb R$ we have
\begin{equation}
\label{eq:mea}
- \int_0^T \chi'(t) \entropy_\psi(\rho) \cd t
\le \int_0^T \chi(t)
\entropyproduction_\psi(\rho) \cd t
;
\end{equation}
\item
the initial entropy satisfies
\begin{equation}
\label{eq:init-entropy}
\esssup_{t > 0} \entropy_\psi(\rho(t)) \le \entropy_\psi(\rho^0);
\end{equation}
\item
$\rho$ satisfies the lower $L^1$-bound
\begin{equation}
\label{eq:lower}
\| \rho(t) \|_{L^1(\Omega)} \ge \| \min(\rho^0, \rho_\infty) \|_{L^1(\Omega)} 
\quad \text{a.~a.\ } t > 0
;
\end{equation}
\item
$\rho$ exponentially converges to~$\rho_\infty$ in the sense of entropy:
\begin{equation}
\label{eq:convergence-a}
\entropy_\psi(\rho(t)) \le \entropy_\psi(\rho^0) e^{-\gamma_\psi t} \quad 
\text{a.~a.\ } t > 0
,
\end{equation}
where $\gamma_\psi > 0$ can be chosen uniformly over initial data satisfying
\begin{equation}
\label{eq:convergence-b}
\| \min(\rho^0, \rho_\infty) \|_{L^1(\Omega)} \ge c, \quad 
\entropy_\psi(\rho^0) \le C
\end{equation}
with some $c, C > 0$;
\item
for any $p\in[2,+\infty)$, 
\begin{equation}
\label{eq:convergence-c}
\|\rho(t)-\rho_\infty\|_{L^p(\Omega)} \le   e^{-\gamma_p t}\,\left(1+\frac {\sup ~\rho_\infty}{\inf ~\rho_\infty}\right)\,\|\rho^0-\rho_\infty\|_{L^p(\Omega)}\quad 
\text{a.~a.\ } t > 0
,
\end{equation}
where $\gamma_p > 0$ can be chosen uniformly over initial data satisfying
\begin{equation}
\label{eq:convergence-b2}
\| \min(\rho^0, \rho_\infty) \|_{L^1(\Omega)} \ge c, \quad 
\|\rho^0\|_{L^p(\Omega)}^p \le C.
\end{equation}
\end{enumerate}
\end{theorem}
\begin{proof}
For the proof of existence, the approximating procedure used in~\cite{KV17} is 
still applicable in the current setting.  As a matter of fact, the existence 
result in~\cite{KV17} requires that~$|f(x, \xi)|$ is either large or does not depend on~$x$ 
when~$\xi$ is near $0$ or near $+\infty$.  A similar requirement 
was imposed for large~$\xi$.  However, these assumptions are only needed in 
order to ensure that any $u \in L_+^\infty(\Omega)$ can be bounded from above by a 
function $u_c \colon \Omega \to \mathbb R$  satisfying $f(x, u_c(x)) \equiv cst$ 
and that $u$ can be bounded from below by another such function provided 
that~$u$ is uniformly bounded away from~$0$.  This is still the case in the current 
setting.  Indeed, assume for simplicity that~$u$ is continuous on~$\overline 
\Omega$.  Set $c = \max_\Omega g(u/\rho_\infty)$ and put $u_c = \rho_\infty 
g^{-1}(c)$, then clearly $f(x, u_c(x)) = - g(u_c(x)/\rho_\infty) = -c$; 
moreover, it follows from the monotonicity of~$g$ that $u \le u_c$, as 
required.  The existence of a lower bound is proved in a similar way, 
cf.~\cite[Remark~3.4]{KV17}.

Inequality~\eqref{eq:init-entropy} is proved in the same way as the analogous 
inequality in~\cite{KV17}.

We prove that the solution constructed as in~\cite{KV17} satisfies \eqref{eq:mea}.  
To this end it suffices to check that this inequality is preserved under the 
passage to the limit.  Specifically, assume that smooth enough approximate 
solutions $\{\rho_n\}$ are uniformly bounded in~$L^\infty(Q_T)$ and converge 
to $\rho$ a.~e.\ in~$Q_T$, while
\begin{equation*}
G_n := G(r_n) \to G(r) \qquad \text{weakly in $L^2(\Omega)$}
.
\end{equation*}
By the Lebesgue Dominated Convergence we have
\begin{gather}
\entropy_\psi(\rho_n) \to \entropy_\psi(\rho)
,
\label{eq:tmp9}
\\
\int_\Omega r_n g(r_n) \psi'(r_n) \cd \rho_\infty
\to
\int_\Omega r g(r) \psi'(r) \cd \rho_\infty
\label{eq:tmp10}
.
\end{gather}
Arguing as in~\cite[proof of Theorem~3.9]{KV17} and, in particular, taking into 
account that $\nabla G = 0$ a.~e.\ on the set $\{(x, t) \in Q_T \colon r = 0 
\}$ and $\nabla G_n = 0$ a.~e.\ on the set $\{(x, t) \in Q_T \colon r_n = 0 
\}$, we conclude that for any 
$\delta > 0$ we have
\begin{multline*}
\iint\limits_{\{(x, t) \in Q_T \colon r > 0 \}}
\frac{\chi(t)\psi''(r)}{\max(r, \delta) g'(r)}
|\nabla G|^2
\cd \rho_\infty \cd t
\\ \le
\liminf_{n \to \infty}
\iint\limits_{\{(x, t) \in Q_T \colon r_n > 0 \}}
\frac{\chi(t)\psi''(r_n)}{\max(r_n, \delta) g'(r_n)}
|\nabla G_n|^2
\cd \rho_\infty \cd t
\\ 
\le
\liminf_{n \to \infty}
\iint\limits_{\{(x, t) \in Q_T \colon r_n > 0 \}}
\frac{\chi(t)\psi''(r_n)}{r_n g'(r_n)}
|\nabla G_n|^2
\cd \rho_\infty \cd t
,
\end{multline*}
so sending $\delta \to \infty$ and applying Beppo Levy's theorem, we obtain
\begin{equation*}
\iint\limits_{\{(x, t) \in Q_T \colon r > 0 \}}
\frac{\chi(t)\psi''(r)}{r g'(r)}
|\nabla G|^2
\cd \rho_\infty \cd t
\le
\liminf_{n \to \infty}
\iint\limits_{\{(x, t) \in Q_T \colon r_n > 0 \}}
\frac{\chi(t)\psi''(r_n)}{r_n g'(r_n)}
|\nabla G_n|^2
\cd \rho_\infty \cd t
\end{equation*}
or, equivalently,
\begin{multline*}
\iint\limits_{\{(x, t) \in Q_T \colon r > 0 \}}
\chi(t) r g'(r) \psi''(r) |\nabla r|^2
\cd \rho_\infty \cd t\\
\le
\liminf_{n \to \infty}
\iint\limits_{\{(x, t) \in Q_T \colon r_n > 0 \}}
\chi(t) r_n g_n'(r) \psi''(r_n) |\nabla r_n|^2
\cd \rho_\infty \cd t
.
\end{multline*}
Combining this with~\eqref{eq:tmp9} and~\eqref{eq:tmp10}, we 
obtain~\eqref{eq:mea}.

We now prove the exponential convergence of the solution to the 
steady state. Let $\rho$ be a weak solution of~\eqref{eq:p1}--\eqref{eq:p3} with the initial 
data satisfying~\eqref{eq:convergence-b}.  Let $U \subset L^\infty_+$ be the set of functions such that 
for any $u \in U$, we have $G(u/\rho_\infty) \in H^1(\Omega)$ and 
$\|u\|_{L^1(\Omega)} \ge c$, $\entropy_\psi(u) \le C$ with the same~$c$ 
and~$C$ as in \eqref{eq:convergence-b}.  By Theorem~\ref{th:eep} we have the 
entropy-entropy production inequality~\eqref{eq:eep-b} for~$U$.
It follows from the bounds~\eqref{eq:init-entropy} and~\eqref{eq:lower}
that $\rho(t) \in U$ for a.~a.\ $t > 0$.  Combining the entropy dissipation 
and entropy-entropy production inequalities, we get
\begin{equation*}
\partial_t \entropy_\psi(\rho(t)) \le - C_U^{-1} \entropy_\psi(\rho(t))
\end{equation*}
in the sense of measures.  Set $\gamma_\psi = C_U^{-1}$ and $\phi(t) = 
\entropy_\psi(\rho(t)) e^{\gamma_\psi t}$.  It is easy to check that that 
$\partial_t \phi(t) \le 0$ in the sense of measures, whence $\phi$ a.~e.\ 
coincides with a nonincreasing function.  Moreover,
\begin{equation*}
\esssup_{t > 0} \phi(t)
= \esslimsup_{t \to 0} \phi(t)
= \esslimsup_{t \to 0} \entropy_\psi(\rho(t)) e^{\gamma_\psi t}
\le \entropy_\psi(\rho^0)
\end{equation*}
by virtue of~\eqref{eq:init-entropy}, so $\phi(t) \le \entropy_\psi(\rho^0)$ for 
a.~a.\ $t > 0$, which implies~\eqref{eq:convergence-a}. 

We will now use \eqref{eq:convergence-a} with $\psi(s)=|s-1|^p$, which is a $C^2$-function for $p\geq 2$, and satisfies the assumptions~\eqref{eq:psi1}--\eqref{eq:infty2}. We immediately get
\begin{multline}
\label{eq:convergence-d}
\|\rho(t)-\rho_\infty\|_{L^p(\Omega)} \le (\sup ~\rho_\infty)^{(p-1)/p}[\entropy_\psi(\rho(t))]^{1/p} \\ \le  (\sup ~\rho_\infty)^{(p-1)/p}[\entropy_\psi(\rho^0)]^{1/p} e^{-\gamma_\psi t/p}\\ \le  \left(\frac {\sup ~\rho_\infty}{\inf ~\rho_\infty}\right)^{(p-1)/p}\|\rho^0-\rho_\infty\|_{L^p(\Omega)}e^{-\gamma_p t}\\\le  \left(1+\frac {\sup ~\rho_\infty}{\inf ~\rho_\infty}\right)\|\rho^0-\rho_\infty\|_{L^p(\Omega)}e^{-\gamma_p t}
,
\end{multline} where $\gamma_p=\gamma_\psi/p$. Uniform boundedness of $\|\rho^0\|_{L^p}^p$ implies a bound on $\entropy_\psi(\rho^0)$.
\end{proof}

\section{Inequality}
\label{sec:ineq}
In this section we prove a refined version of our unbalanced convex Sobolev inequality in the smooth case.
\begin{theorem}
\label{th:psi}
Assume \assumptions.  Let $U \in C^\infty_+(\Omega)$ be such that
\begin{gather*}
\inf \big \{ \| \rho \|_{L^1(\Omega)} \colon \rho \in U \big \} > 0
,
\\
\sup \{ \entropy_\psi(\rho) \colon \rho \in U \} < \infty
.
\end{gather*}
Then there exist constants  (independent of $\rho$) $C > 0$, $0 < \alpha < \beta < \infty$, such that
\begin{equation}
\label{eq:eep}
\entropy_\psi(\rho) \le C
\left(
\int_\Omega
r g(r) \psi'(r) \cd \rho_\infty
+
\int_{[\alpha < r < \beta]}
rg'(r) \psi''(r) |\nabla r|^2 \cd \rho_\infty
\right)
\quad (\rho \in U)
.
\end{equation}
\end{theorem}

The proof of Theorem \ref{th:psi} is based on the next two lemmas. 

\begin{lemma}
\label{lem:alphabeta}
Fix $0 < \alpha < \beta < 1$. Then
\begin{multline}
\label{eq:alphabeta}
\big| [\alpha < r < \beta] \big|
\int_{[\alpha < r < \beta]} rg'(r) \psi''(r) |\nabla r|^2 \, d\rho_\infty
\\
\ge
C_{\alpha\beta}
\min\left(
\big | [r \le \alpha] \big|^{2(d-1)/d},
\big | [r \ge \beta] \big|^{2(d-1)/d}
\right)
\end{multline}
\end{lemma}
\begin{proof}
If the minimum on the right-hand side vanishes, there is nothing to prove.  
Otherwise the set $[\alpha < r < \beta]$ has nonzero measure. In what follows, we use some facts from geometric 
measure theory, which can be found in \cite{Mag12}.  The relative perimeter of a Lebesgue measurable set $A$ of locally finite 
perimeter with respect to $\Omega$
is 
$
P(A; \Omega) = |\mu_A| (\Omega)
,
$
where $\mu_A:=\nabla 1_A$ is the Gauss-Green measure associated with~$A$. 
The support of~$\mu_A$ is contained in the topological boundary 
of~$A$.

 We have:
\begin{multline}
\int_{[\alpha < r < \beta]} r g'(r) \psi''(r) |\nabla r|^2 \, d\rho_\infty
\ge \inf_\Omega \rho_\infty \min_{s \in [\alpha, \beta]} (sg'(s) \psi''(s))
\int_{[\alpha < r < \beta]} |\nabla r|^2 \, dx
\\
\ge \frac{ \inf_\Omega \rho_\infty \min_{s \in [\alpha, \beta]} (s g'(s)
\psi''(s)) }{\big| [\alpha < r < \beta] \big|}
\left(
\int_{[\alpha < r < \beta]} |\nabla r| \, dx
\right)^2
\label{eq:tmp2}
\end{multline} 

The last integral is the variation of~$r$ over~$[\alpha < r < \beta]$, which 
can be computed using the coarea formula:
\begin{align}
\int_{[\alpha < r < \beta]} |\nabla r| \, dx
& =\int_{-\infty}^\infty P([r < t]; [\alpha < r < \beta]) \, dt
\notag
\\
& =\int_\alpha^\beta P([r < t]; [\alpha < r < \beta]) \, dt
\notag
\\
& =\int_\alpha^\beta P([r < t]; \Omega) \, dt
,
\label{eq:tmp3}
\end{align}
where we first use the observation that the support of the Gauss--Green 
measure associated with $[r < t]$ is disjoint with $[\alpha < r < \beta]$ 
whenever $t \le \alpha$ or $t \ge \beta$, and then we notice that if $\alpha < 
t < \beta$, then the part of the support of the Gauss--Green measure of $[r < 
t]$ lying in~$\Omega$ is contained in $[\alpha < r < \beta]$.

Invoking the relative isoperimetric inequality \eqref{eq:iso}, we estimate
\begin{equation*}
P([r < t]; \Omega) \ge C_\Omega \min\left( \big | [r < t] \big|^{(d-1)/d}, 
\big| \Omega \setminus [r < t] \big|^{(d-1)/d}
\right)
\end{equation*}
and since for $t \in (\alpha, \beta)$ we have
\begin{equation*}
[r \le \alpha] \subset [r < t] \subset [r < \beta] = \Omega \setminus [r \ge 
\beta]
\end{equation*}
we see that
\begin{equation*}
P([r < t]; \Omega)
\ge C_\Omega
\min\left(
\big | [r \le \alpha] \big|^{(d-1)/d},
\big | [r \ge \beta] \big|^{(d-1)/d}
\right)
\end{equation*}
Combining this estimate with~\eqref{eq:tmp2} and~\eqref{eq:tmp3}, we 
obtain~\eqref{eq:alphabeta}.
\end{proof}
\begin{lemma}
\label{lem:algterm}
Given $\varepsilon > 0$, there exists $C_\varepsilon > 0$ such that
\begin{equation} \label{eq:lem}
\psi(s) \le C_\varepsilon s g(s) \psi'(s) \quad (s \ge \varepsilon)
.
\end{equation}
\end{lemma}
\begin{proof}
Applying L'Hôpital's rule for $\liminf$, and remembering that $g$ is an increasing function, we obtain
\begin{equation}
\liminf_{s \to \infty}
\frac{s g(s) \psi'(s)}{\psi(s)} \geq \liminf_{s \to \infty}\left(
g(s)+s g'(s)+\frac{s g(s) \psi''(s)}{\psi'(s)}\right)\ge \lim_{s \to \infty} g(s)> 0,
\label{eq:infty} \end{equation}
\begin{equation}
\label{eq:tmp1}
\liminf_{s \to 1}\frac{sg(s)\psi'(s)}{\psi(s)}=\liminf_{s \to 1}\frac{g(s)\psi'(s)}{\psi(s)}
\ge
\liminf_{s \to 1}
\left(
g'(s)
+
\frac{g(s)\psi''(s)}{\psi'(s)}
\right)
\ge
g'(1) > 0.
\end{equation}
 In \eqref{eq:infty} and \eqref{eq:tmp1} we have used the fact that for $s \ne 1$, the signs of~$g(s)$ 
and~$\psi'(s)$ coincide, while $\psi''(s) > 0$. Obviously, \eqref{eq:infty} and \eqref{eq:tmp1} imply \eqref{eq:lem}.
\end{proof}


\begin{proof}[Proof of Theorem~\ref{th:psi}]

We claim that there exists $\beta > 0$ such that
\begin{equation}
\label{eq:defbeta}
\delta := \inf_{\rho \in U} \big| [ r \ge \beta ] \big| > 0
\end{equation}
Indeed, it follows from~\eqref{eq:infty2} (L'Hôpital's rule) that
\begin{equation*}
\lim_{s \to \infty} \frac{\psi(s)}{s} = \infty
.
\end{equation*}
As the entropy~$\entropy_\psi$ is bounded on~$U$, by de la Vallée Poussin's
theorem the set~$U$ is uniformly integrable.  Put
\begin{equation*}
m = \frac{1}{2|\Omega|} \inf_{\rho \in U} \| \rho \|_{L^1(\Omega)}
;
\end{equation*}
for any $\rho \in U$ we have
\begin{equation*}
2|\Omega|m \le \| \rho \|_{L^1(\Omega)}
= \int_{[\rho < m]} \rho \, dx
+ \int_{[\rho \ge m]} \rho \, dx
\le |\Omega| m + \omega_U\left(\big|[\rho \ge m] \big|\right)
,
\end{equation*}
where $\omega_U$ is the modulus of integrability of~$U$.  Hence
\begin{equation*}
\omega_U\left(\big|[\rho \ge m] \big|\right) \ge |\Omega| m
,
\end{equation*}
which clearly implies a lower bound on $\big|[\rho \ge m]\big|$ and a fortiori 
on $\big|[r \ge \beta]\big|$ with $\beta = \frac m {\sup \rho_\infty}$.

Clearly, there is no loss in generality in assuming $\beta < 1$ 
in~\eqref{eq:defbeta}.



In what follows we fix $\alpha$ and $\beta$ such that $0 < \alpha < \beta < 1$ 
and $\beta$ satisfies~\eqref{eq:defbeta}.  Denote
\begin{gather*}
\sigma := \big| [r \le \alpha ] \big|,
\\
\tau := \big| [\alpha < r < \beta ] \big|
\end{gather*}
and also
\begin{equation*}
D_{\alpha\beta}\entropy_\psi(\rho) :=
\int_\Omega
r g(r) \psi'(r) \cd \rho_\infty
+
\int_{[\alpha < r < \beta]}
rg'(r) \psi''(r) |\nabla r|^2 \cd \rho_\infty
.
\end{equation*}

Assume for now that $\sigma > 0$.
Using Lemma~\ref{lem:alphabeta}, we have
\begin{align*}
D_{\alpha\beta}\entropy_\psi(\rho)
&
\ge
\int_{[\alpha < r < \beta]}
rg(r) \psi'(r)
\, d \rho_\infty
+ \int_{[\alpha < r < \beta]} rg'(r) \psi''(r)
|\nabla r|\, d \rho_\infty
\\
&
\ge \left(\min_{s \in [\alpha, \beta]} sg(s) \psi'(s)\right)
\tau
+ C_{\alpha\beta} \frac 1\tau \min\left(\sigma^{2(d-1)/d}, \big|[r \ge 
\beta]\big|^{2(d-1)/d}\right)
.
\end{align*}
Taking into account~\eqref{eq:defbeta}, we can write
\begin{equation*}
D_{\alpha\beta}\entropy_\psi(\rho)
\ge
\frac c2
\left(
\tau
+\frac{\min(\sigma^{2(d-1)/d}, \delta^{2(d-1)/d})}{\tau}
\right)
\end{equation*}
with $c$ independent of~$\rho$.  Estimating
\begin{equation*}
\tau +\frac{\min(\sigma^{2(d-1)/d}, \delta^{2(d-1)/d})}{\tau}
\ge
2\min(\sigma^{(d-1)/d}, \delta^{(d-1)/d})
,
\end{equation*}
we obtain
\begin{equation}
\label{eq:tmp4}
D_{\alpha\beta}\entropy_\psi(\rho) \ge
c\min(\sigma^{(d-1)/d}, \delta^{(d-1)/d})
.
\end{equation}
If $\sigma = 0$, this estimate trivially holds with any $c$.
Since $\sigma$ is a priori bounded from above by~$|\Omega|$, \eqref{eq:tmp4} implies that
\begin{equation}
\label{eq:tmp5}
\sigma \leq C \min\left(\frac \sigma {|\Omega|^{1/d}}, \frac{\delta^{(d-1)/d}\sigma}{|\Omega|} \right)\leq C\min(\sigma^{(d-1)/d}, \delta^{(d-1)/d})\leq C D_{\alpha\beta}\entropy_\psi(\rho).
\end{equation}

Evoking Lemma~\ref{lem:algterm}, we obtain
\begin{align*}
\entropy_\psi(\rho)
& =
\int_{[r > \alpha]} \psi(r) \, d\rho_\infty
+
\int_{[r \le \alpha]} \psi(r) \, d\rho_\infty
\\
&
\le C_\alpha
\int_{[r > \alpha]} r \psi'(r) g(r) \, d\rho_\infty
+
\psi(0) \int_{[r \le \alpha]} \, d\rho_\infty
\\
& \le C_\alpha
D_{\alpha\beta}\entropy_\psi(\rho)
+
C_0 \big| [ r \le \alpha ] \big|
\\
& \le C D_{\alpha\beta}\entropy_\psi(\rho) + C\sigma
.
\end{align*}
Using~\eqref{eq:tmp5} to estimate $\sigma$ by $D_{\alpha\beta}\entropy_\psi$, 
we obtain~\eqref{eq:eep}
\end{proof}
\appendix
\section{Reverse Fatou's Lemma for products}

\begin{lemma}
\label{lem:fatou}
Let $(S, \Sigma, \mu)$ be a measure space.  Suppose that $\{f_n\}$ is bounded 
in~$L^\infty(S, \mu)$ and $\{g_n\}$ converges to a nonnegative limit $g$ 
in~$L^1(S, \mu)$.  Then
\begin{equation}
\label{eq:fatou}
\limsup_{n \to \infty} \int_S f_n g_n \cd \mu
\le
\int_S \left( \limsup_{n \to \infty} f_n \right) g \cd \mu
.
\end{equation}
\begin{proof}
As we have $|f_n g| \le (\sup_n \|f_n\|) g$, we can use Reverse Fatou's Lemma 
obtaining
\begin{equation}
\label{eq:tmp11}
\limsup_{n \to \infty} \int_S f_n g \cd \mu
\le 
\int_S \left(\limsup_{n \to \infty} f_n  g\right)  \cd \mu=
\int_S \left(\limsup_{n \to \infty} f_n \right) g  \cd \mu
.
\end{equation}
 Further, it is 
clear that
\begin{equation}
\label{eq:tmp12}
\lim_{n \to \infty} \int_S f_n(g_n - g) \cd \mu
= 0
.
\end{equation}
Using~\eqref{eq:tmp11} and~\eqref{eq:tmp12} we obtain
\begin{align*}
\limsup_{n \to \infty} \int_S f_n g_n &=
\limsup_{n \to \infty} \left(
\int_S f_n g \cd \mu
+ \int_S f_n(g_n - g) \cd \mu
\right)
\\&
=
\limsup_{n \to \infty} \int_S f_n g \cd \mu
+
\lim_{n \to \infty} \int_S f_n(g_n - g) \cd \mu
\\&
\le
\int_S \left(\limsup_{n \to \infty} f_n\right) g  \cd \mu
,
\end{align*}
as claimed.
\end{proof}
\end{lemma}

\subsection*{Acknowledgment}
The research was partially supported by the Portuguese Government through FCT/MCTES and by the ERDF through PT2020 (projects UID/MAT/00324/2019, PTDC/MAT-PUR/28686/2017 and TUBITAK/0005/2014).

\subsection*{Conflict of interest statement} We have no conflict of interest to declare.

\end{document}